\documentclass[10pt,reqno]{amsart}
\usepackage[T1]{fontenc}
\usepackage[latin1]{inputenc}
\usepackage{color}
\usepackage{amsthm}
\usepackage{amstext}
\usepackage{amssymb}
\usepackage[all]{xy}
\usepackage[unicode=true,pdfusetitle,
 bookmarks=false,
 breaklinks=false,pdfborder={0 0 1},backref=false,colorlinks=true]
 {hyperref}

\makeatletter
\numberwithin{equation}{section}
\numberwithin{figure}{section}
\usepackage{enumitem}		
\theoremstyle{plain}
\newtheorem{thm}{\protect\theoremname}
  \theoremstyle{definition}
  \newtheorem{defn}[thm]{\protect\definitionname}
  \theoremstyle{definition}
  \newtheorem{example}[thm]{\protect\examplename}
  \theoremstyle{remark}
  \newtheorem{rem}[thm]{\protect\remarkname}
  \theoremstyle{plain}
  \newtheorem{prop}[thm]{\protect\propositionname}
  \theoremstyle{plain}
  \newtheorem{cor}[thm]{\protect\corollaryname}

\@ifundefined{date}{}{\date{}}

\usepackage{pdfsync}
\usepackage{graphics}
\usepackage{epstopdf}
\usepackage[all]{xy}
\usepackage{amscd}

\newcommand{\xyR}[1]{ \makeatletter
\xydef@\xymatrixrowsep@{#1} \makeatother} 
\newcommand{\xyC}[1]{ \makeatletter
\xydef@\xymatrixcolsep@{#1} \makeatother} 
\entrymodifiers={++[ ][F]} 

\newcommand{\ra}{\longrightarrow}

\newcommand{\field}[1]{\mathbb{#1}}
\newcommand{\R}{\field{R}} 
\newcommand{\N}{\field{N}} 
\newcommand{\Z}{\ensuremath{\mathbb{Z}}} 

\newcommand{\F}{\mathcal{F}}

\renewcommand{\phi}{\varphi} 

\DeclareMathOperator*{\colim}{colim} 
\newcommand{\Coo}{\mbox{\ensuremath{\mathcal{C}}}^{\infty}} 
\DeclareMathOperator{\Set}{{\bf Set}} 
\newcommand{\FDiff}{\ext{\Coo}} 
\newcommand{\Top}{{\bf Top}} 
\newcommand{\Diff}{{\bf Diff}} 

\newcommand{\FR}{{{}^\bullet\R}} 
\newcommand{\st}[1]{{\makebox[0pt][r]{\phantom{$#1$}}^{\circ} #1}} 
\newcommand{\ext}[1]{{}^\bullet #1} 

\newcommand{\blank}{-}
\newcommand{\CSh}{{\mathfrak{C}\mathrm{Sh}}} 

\newcommand{\So}{\mathcal{S}}
\newcommand{\FU}{\ext{U}}
\newcommand{\FV}{\ext{V}}
\newcommand{\FtV}{\ext{\tilde{V}}}
\newcommand{\FW}{\ext{W}}
\newcommand{\FX}{\ext{X}}
\newcommand{\FY}{\ext{Y}}
\newcommand{\FZ}{\ext{Z}}
\newcommand{\FA}{\ext{A}}

\newcommand{\FM}{\ext{M}}
\newcommand{\Ff}{\ext{f}}
\newcommand{\Fg}{\ext{g}}

\newcommand{\Fp}{\ext{p}}
\newcommand{\Fq}{\ext{q}}
\newcommand{\Lf}{{{}_{\bullet} f}}
\newcommand{\Lg}{{{}_{\bullet} g}}
\newcommand{\LX}{{{}_{\bullet} X}}
\newcommand{\LY}{{{}_{\bullet} Y}}

\makeatother

  \providecommand{\corollaryname}{Corollary}
  \providecommand{\definitionname}{Definition}
  \providecommand{\examplename}{Example}
  \providecommand{\propositionname}{Proposition}
  \providecommand{\remarkname}{Remark}
\providecommand{\theoremname}{Theorem}

\begin{document}

\title{The Fermat Functors\\
{\footnotesize{}Part I: The theory}}

\author{Enxin Wu}

\thanks{E.~Wu has been $\text{supp}$orted by grant P25311-N25 of the Austrian
Science Fund FWF}

\address{\textsc{Faculty of Mathematics, University of Vienna, Austria, Oskar-Morgenstern-Platz
1, 1090 Wien, Austria}}

\email{\texttt{enxin.wu@univie.ac.at}}
\begin{abstract}
In this paper, we use some basic quasi-topos theory to study two functors:
one adding infinitesimals of Fermat reals to diffeological spaces
(which generalize smooth manifolds including singular spaces and infinite
dimensional spaces), and the other deleting infinitesimals on Fermat
spaces. We study the properties of these functors, and calculate some
examples. These serve as fundamentals for developing differential
geometry on diffeological spaces using infinitesimals in a future
paper.
\end{abstract}

\maketitle
\tableofcontents{}

\section{Introduction}

Using infinitesimals to study geometry goes back to I. Newton or even
earlier, as one of the motivations for developing calculus, and hence
the start of the modern mathematics. Although infinitesimal theory
was not rigorous at the beginning, the intuitive idea behind it was
so enlightening that a great amount of work at that time by mathematicians
like L. Euler, J.-L. Lagrange, etc, were influenced by that. It was
A.-L. Cauchy who made the definition of limit rigorous using the epsilon-delta
language. Since then infinitesimal theory gradually left the main
stream of mathematics.

On the other hand, many concepts in geometry came from intuitive infinitesimal
considerations, for example, tangent vectors, vector fields, Lie groups,
Lie algebras, connections, curvature, etc. Many modern formulations
of these concepts leave very little trace of their original ideas,
but they are very convenient for doing computations. In other words,
there is a step from translating geometric ideas using infinitesimals
to the modern formulations, and most of time, this step is left as
a gap in most literature, especially for students start to learn this
field. It is always a great hope that infinitesimal theory could be
made rigorous and enter differential geometry for the real content. 

Going back to the rigor of infinitesimal theory, nowadays, there are
a few such theories available on the market. Two of the most developed
ones are Non-Standard Analysis (NSA; see for example~\cite{Ro96})
and Synthetic Differential Geometry (SDG; see for example~\cite{Ko06,La96}).
The infinitesimals in SDG are nilpotent%
\footnote{More precisely, the square of any infinitesimal number in SDG is 0.%
}, while those in NSA are not. In the smooth manifold%
\footnote{By a smooth manifold in this paper, we always mean it to be finite
dimensional, second countable, Hausdorff, and without boundary.%
} case, a tangent space at a point is the first-order approximation
of the manifold. That is, if we embed the smooth manifold in some
Euclidean space, and assume that the local defining function of the
manifold around that point is $f$, then the Taylor expansion of $f$
up to order $1$ is the tangent space there. This example implies
that we can use nilpotent infinitesimals to do differential geometry,
i.e., if the infinitesimal incremental of the variables of $f$ already
has the property that the multiplication of any two of them is $0$,
then the tangent space is exactly the Taylor expansion at that point.
One can also imagine that the higher order geometric structures such
as jets can be characterized using higher order infinitesimals. In
order to axiomatize first-order approximation, SDG has a very strong
axiom called the Kock-Lawvere axiom. This axiom requires the (commutative
unital) ring with infinitesimals to satisfy an affine condition for
every function from infinitesimals to the ring, and in the framework
of classical logic, the only such ring is the trivial ring. In other
words, the whole theory of SDG is built upon a new world called intuitionistic
logic.

The theory of Fermat reals introduced by P. Giordano in~\cite{Gio10}
is another infinitesimal theory, where every infinitesimal is nilpotent,
and the theory is compatible with classical logic; see Section~\ref{sec:BasicFermatReals}
for a brief summary of the basics of this theory. It is not hard to
redo many classical constructions of differential geometry on smooth
manifolds using Fermat reals; some of them have already been explored
in~\cite{Gio09}, and more will be presented systematically in a
following paper. 

Note that many spaces other than smooth manifolds arise naturally
and frequently in geometry, for example, smooth manifolds with boundary
or corners, singular orbit spaces of Lie groups acting on smooth manifolds
(in particular, orbifolds), function spaces between smooth manifolds,
diffeomorphism groups of smooth manifolds, etc. These spaces are usually
studied separately in the literature. There are generalizations of
smooth manifolds which contain (some of) them. Diffeology is one of
such generalizations, introduced by J.-M. Souriau in~\cite{So80,So84}.
A standard textbook is~\cite{Igl13}. Briefly, a diffeological space
is a set together with specified functions from open subsets of $\R^{n}$
for all $n$ to this set, satisfying three simple axioms. These axioms
declare when a function from an open subset of a Euclidean space to
this set is ``smooth''. A typical non-trivial and important example
is an irrational torus, which cannot be characterized by (continuous)
maps from this space to Euclidean spaces. Moreover, there is a quasi-topos
approach to diffeology (\cite{Ba-Ho11}), and this idea has been extended
to Fermat reals (\cite{Gio11}), called Fermat spaces. 

An approach of adding infinitesimals on diffeological spaces has been
tried in~\cite{Gio09}, which uses maps from diffeological spaces
to Euclidean spaces to identify little-oh polynomials (a pre-model
for infinitesimals) on this diffeological space; see~\cite[Chapter~8]{Gio09}.
The theory goes well with smooth manifolds, but not with general diffeological
spaces%
\footnote{For example, it is easy to check by definition that it turns an irrational
torus to a trivial Fermat space.%
}. This leads the author to think of a very different approach of adding
infinitesimals. More precisely, since diffeological spaces are concrete
sheaves over the Souriau site, and Fermat spaces are concrete sheaves
over the Fermat site (Example~\ref{exa:concrete-sheaves}), to find
natural relationship between them, it is enough to find natural functors
between the two sites. In this way, we not only get a definition of
the adding infinitesimal functor (Proposition~\ref{prop:Fermat-extension})
which is different from the one presented in~\cite{Gio09} (Proposition~\ref{prop:difference}),
but also obtain its left inverse, called the deleting infinitesimal
functor (Propositions~\ref{prop:def-delete} and~\ref{prop:delete}).
Almost every property of the adding infinitesimal functor in~\cite{Gio09}
holds in this new definition, with most restrictive conditions removed
(see Subsection~\ref{sub:add}). In Subsection~\ref{sub:why}, we
discuss the comparison of another natural adding infinitesimal functor%
\footnote{It uses another canonical concrete site; see Remark~\ref{rem:site}.
Indeed, we used this concrete site in the definition of the deleting
infinitesimal functor.%
} with the current one, and explain why the current one is better.
Finally in Subsection~\ref{sub:Calculations}, we do a few calculations,
and show that in general the calculation is not easy%
\footnote{More precisely, we show in Example~\ref{exa:not-commute-colimits}
that the adding infinitesimal functor does not commute with arbitrary
colimits.%
}. All of these will serve as fundamentals for developing differential
geometry on diffeological spaces in a future paper.

\medskip{}

I would like to thank P. Giordano for suggesting this project.

\section{\label{sec:BasicFermatReals}Basic of Fermat reals}

Fermat reals were introduced by P. Giordano in \cite{Gio09,Gio10,Gio11,Gi-Ku13}.
Let us review the basic theory here; see these references for detailed
proof of these results.

\medskip{}

Let $U$ be an open subset of $\R^{n}$. We define $U_{0}[t]$, the
\emph{little-oh polynomials} on $U$, to be the set of functions $x:[0,\epsilon)\ra U$
for some (not fixed) $\epsilon\in\R_{>0}$ with the property that
\[
\|x(t)-r-\sum_{i=1}^{k}\alpha_{i}t^{a_{i}}\|=o(t)\;\text{i.e.,}\;\lim_{t\ra0}\frac{\|x(t)-r-\sum_{i=1}^{k}\alpha_{i}t^{a_{i}}\|}{t}=0
\]
for some $r\in U$, $k\in\N$, $\alpha_{i}\in\R^{n}$ and $a_{i}>0$.
Two little-oh polynomials $x$ and $y$ are called equivalent if $x(0)=y(0)$
and $x(t)-y(t)=o(t)$. This is an equivalence relation on $U_{0}[t]$,
and the quotient set is denoted by $\FU$. As a consequence, every
element in $\FU$ has a unique representing little-oh polynomial of
the form 
\begin{equation}
y(t)=\st{y}+\sum_{i=1}^{l}\beta_{i}t^{b_{i}}\label{eq:decomposition}
\end{equation}
 for some $\st{y}(:=y(0))\in U$, $l\in\N$, $\beta_{i}\in(\R^{n}\setminus\{0\})$
and $0<b_{1}<b_{2}<\cdots<b_{l}\leq1$, defined on $[0,\delta)$ for
some maximum $\delta\in\R_{>0}\cup\{\infty\}$.%
\footnote{Careful reader will notice that there are two main differences between
the presentation here and the one in the existing references. One
is, we use germs at $0$ for little-oh polynomials, since sometimes,
such functions are not necessarily globally defined, and another one
is the expression of the unique representing little-oh polynomials
without using the notation changes: $t^{b}\leftrightarrow dt_{1/b}$,
since from my opinion, the use of $t^{b}$ is closer to the traditional
way of viewing such functions as a kind of ``polynomials'', and much
easier for doing computations, etc. %
} We call this the \emph{decomposition} of the element $[y]$, $\st{y}$
the \emph{standard part}, and we define $\omega([y]):=\frac{1}{b_{1}}$
the \emph{order} of $[y]$. For convenience, we sometimes use a similar
form of $y(t)$ as \eqref{eq:decomposition} but allowing $\beta_{i}=0$,
and we call such a form a \emph{quasi-decomposition} of $[y]$. From
now on, we write elements in $\FU$ by $y$ instead of $[y]$ whenever
there is no confusion.

Given a finite set of open subsets $\{U_{i}\}_{i\in I}$ of Euclidean
spaces, $^{\bullet}(\prod_{i\in I}U_{i})$ naturally bijects $\prod_{i\in I}\FU_{i}$.
Therefore, we do not distinguish $^{\bullet}(\R^{n})$ and $(\FR)^{n}$,
and write it as $\FR^{n}$. We can also identify $\FU$ as a subset
of $\FR^{n}$ by $\FU=\{x\in\FR^{n}\mid\st{x}\in U\}$ when $U$ is
an open subset of $\R^{n}$.

There are canonical functions $i_{U}:U\ra\FU$ and $ev_{0}:\FU\ra U$
defined by $i_{U}(u)(t)=u$ and $ev_{0}(x)=\st{x}$, and we have $ev_{0}\circ i_{U}=1_{U}$.
Therefore, $\FU$ is an extension of $U$, and for $x\in\FU$, we
call $\delta x:=x-\st{x}$ the \emph{infinitesimal part} of $x$.
The meaning is clear when $U=\R$: we can give a well ordering on
$\FR$%
\footnote{It is a commutative unital ring under pointwise addition and pointwise
multiplication, called the ring of \emph{Fermat reals.}%
} by $x\leq y$ if $x=\st{x}+\sum_{i=1}^{n}\alpha_{i}t^{a_{i}}$ and
$y=\st{y}+\sum_{i=1}^{n}\beta_{i}t^{a_{i}}$, both in the quasi-standard
form, with $(\st{x},\alpha_{1,}\ldots,\alpha_{n})\leq(\st{y},\beta_{1},\ldots,\beta_{n})$
in the dictionary order, and then $D_{\infty}:=\{x\in\FR\mid\st{x}=0\}=\{x\in\FR\mid-r<x<r\text{ for all }r\in\R_{>0}\}$.
Moreover, every infinitesimal part $\delta x$ of $x\in\FU$ is nilpotent,
i.e., there exists some $m=m(x)\in\N$ such that $(\delta x)^{m}=0$.

$D_{\infty}$ is the unique maximal (prime) ideal of $\FR$. The subsets
$\{0\}$, $D_{a}=\{x\in D_{\infty}\mid\omega(x)<a+1\}$ for all $a\in\R_{>0}\cup\{\infty\}$
and $I_{b}=\{x\in D_{\infty}\mid\omega(x)\leq b\}$ for all $b\in\R_{\geq1}$
are all the ideals of $\FR$. We simply write $D$ for $D_{1}$, called
the set of \emph{first-order infinitesimals}.

On $\FR^{n}$, define $\tau:=\{\FU\mid U\text{ is an open subset of }\R^{n}\}$.
Then $\tau$ is a topology on $\FR^{n}$, called the \emph{Fermat
topology}, since $^{\bullet}(U\cap V)=\FU\cap\FV$ and $^{\bullet}(\cup_{i}U_{i})=\cup_{i}\FU_{i}$.
Without specification, for every subset $A$ of $\FR^{n}$, we always
equip it with the sub-topology of the Fermat topology of $\FR^{n}$.

Let $f:U\ra V$ be a smooth map between open subsets of Euclidean
spaces. Then $\Ff:\FU\ra\FV$ by $\Ff(x)=f\circ x$ is a well-defined
map extending $f$ (called the \emph{Fermat extension of $f$}), i.e.,
we have the following commutative diagram in $\Set$:
\[
\xymatrix{U\ar[d]_{f}\ar[r]^{i_{U}} & \FU\ar[d]_{\Ff}\ar[r]^{ev_{0}} & U\ar[d]^{f}\\
V\ar[r]_{i_{V}} & \FV\ar[r]_{ev_{0}} & V,
}
\]
The calculation of $\Ff(x)=\Ff(\st{x}+\delta x)$ can be done by Taylor's
expansion of $f$ at the point $\st{x}$, using the nilpotency of
$\delta x$. More precisely, if the $(m+1)^{th}$ power of each component
of $\delta x$ is $0$ for some $m\in\N$, then we have 
\[
\Ff(x)=\Ff(\st{x}+\delta x)=\sum_{i\in\N^{m},|i|\leq m}\frac{1}{i!}\frac{\partial^{|i|}f}{\partial x^{i}}(\st{x})\cdot(\delta x)^{i}.
\]
 Therefore, for any open subset $W$ of $V$, we have $^{\bullet}(f^{-1}(W))=(\Ff)^{-1}(\FW)$,
i.e., $\Ff$ is continuous with respect to the Fermat topology.

Note that when $U\neq\emptyset$ and $\dim(V)>0$, not every constant
map $\FU\ra\FV$ is of the form $\Ff$ for some smooth map $f:U\ra V$,
since otherwise $\Ff(u)\in V\subset\FV$ for every $u\in U\subseteq\FU$.
In order to get a concrete site (see next section), we introduce the
following definition:
\begin{defn}
Let $A\subseteq\FR^{n}$ and $B\subseteq\FR^{m}$ be arbitrary subsets.
A function $f:A\ra B$ is called \emph{quasi-standard smooth} if for
every $a\in A$, there exist an open neighborhood $U$ of $\st{a}$
in $\R^{n}$, an open subset $P$ of some Euclidean space, a smooth
map $\alpha:P\times U\ra\R^{m}$ and some fixed point $p\in\ext{P}$,
such that for every $x\in A\cap\FU$, we have 
\[
f(x)=\ext{\alpha}(p,x).
\]

\end{defn}
In particular, every constant map $A\ra B$ and $i_{U}:U\ra\FU$ are
quasi-standard smooth. Moreover, every quasi-standard smooth map is
continuous with respect to the Fermat topology.

\section{\label{s:concrete}Concrete sites and concrete sheaves}

The notion of concrete sites and concrete sheaves goes back to~\cite{Dub79}.
A review of the categories of concrete sheaves, with special attention
to smooth spaces is in~\cite{Ba-Ho11}. We collect some essential
results here and review two examples related to this paper, in order
to unify notations for the following sections. For explicit definitions
and detailed proof of the properties, see~\cite{Dub79}, \cite{Ba-Ho11},
\cite{Joh02} and~\cite[Subsections~1.2 and~2.1]{Wu12}, but we will
not need any of them in this paper.

To be brief, a concrete site is a site with a terminal object, such
that there is a faithful functor from the site to the category $\Set$
of sets and functions (defined using the terminal object), every cover
is jointly surjective on the underlying sets, and every representable
presheaf is actually a sheaf. A concrete sheaf over a concrete site
is a sheaf over this site with an underlying set (as sections over
the terminal object), such that every section is a function between
the underlying sets. Given a concrete site $\mathcal{A}$, the category
$\CSh(\mathcal{A})$ of all concrete sheaves over this site forms
a quasi-topos, i.e., it is like a topos, but with a weak subobject
classifier (that is, it only classifies strong subobjects instead
of all subobjects).

Here are some basic properties of a quasi-topos:
\begin{itemize}
\item It is complete and cocomplete.
\item It is (locally) Cartesian closed.
\item It is locally presentable. 
\end{itemize}
We will make use of the following corollaries a lot in the following
sections:
\begin{enumerate}[leftmargin=*,label=(\roman*),align=left ]
\item The concrete site is canonically a full subcategory of the category
of concrete sheaves over it. By abuse of notation, we use the same
notations to denote objects and morphisms in these categories.
\item Every subset (or quotient set) of a concrete sheaf is canonically
a concrete sheaf.
\item The faithful underlying set functor $|-|:\CSh(\mathcal{A})\ra\Set$
has both left and right adjoints. Therefore, (co)limits in $\CSh(\mathcal{A})$
are the (co)lifting of the corresponding (co)limits for the underlying
sets.
\item Let $\mathcal{A}$ be a concrete site. For any concrete sheaf $X$
over $\mathcal{A}$, write $\mathcal{A}/X$ (called the \emph{plot
category of }$X$) for the overcategory with objects all sections
$p:A\ra X$ and morphisms commutative triangles 
\[
\xymatrix{A\ar[rr]^{f}\ar[dr]_{p} &  & A'\ar[dl]^{p'}\\
 & X
}
\]
where both $p$ and $p'$ are sections and $f$ is a morphism in $\mathcal{A}$.
There is a canonical functor $\mathcal{A}/X\ra\CSh(\mathcal{A})$
sending the above triangle to $f:A\ra A'$, and the colimit of this
functor is $X$. In other words, every concrete sheaf is a colimit
of the representing (concrete) sheaves indexed by the plot category
over it, written as $X=\colim_{A\in\mathcal{A}/X}A$. 
\end{enumerate}
We will mainly focus on the following two examples in this paper:
\begin{example}
\label{exa:concrete-sheaves}\ 
\begin{enumerate}[leftmargin=*,label=(\roman*),align=left]
\item \label{ex:diffeology}\cite[Lemma~4.14 and Proposition~4.15]{Ba-Ho11}
Let $\So$ be the site (called the \emph{Souriau site}) with objects
all open subsets of $\R^{n}$ for all $n\in\N$, morphisms smooth
maps between them, and covers the usual open coverings. Then $\So$
is a concrete site with terminal object $\R^{0}$. The category $\CSh(\So)$
of concrete sheaves over $\So$ is equivalent to the category $\Diff$
of \emph{diffeological spaces} and smooth maps. Isomorphisms in $\Diff$
are called \emph{diffeomorphisms}. For more discussions of diffeological
spaces, see the standard textbook~\cite{Igl13}; for a three-page
concise introduction together with basic notation and terminology,
see \cite[Section~2]{Ch-Si-Wu14}.
\item \label{ex:Fermat}\cite[Section~8.3]{Gio09} Let $\F$ be the site
(called the \emph{Fermat site}) with objects all subsets of $\FR^{n}$
for all $n\in\N$, morphisms quasi-standard smooth maps between them,
and covers the Fermat open coverings. Then $\F$ is a concrete site
with terminal object $\FR^{0}=\R^{0}$.%
\footnote{The reference didn't prove this fact using the language of quasi-topos,
and instead introduced a new terminology called ``a category of figures''.
This fact is indeed an easy consequence of the results proved there.%
} The category of concrete sheaves over $\F$ is denoted by $\FDiff$,
called the category of \emph{Fermat spaces} and Fermat maps.
\end{enumerate}
\end{example}
\begin{rem}
\label{rem:site}We will relate the category $\Diff$ of diffeological
spaces and the category $\FDiff$ of Fermat spaces in next section.
If we define $\F'$ to be the full subcategory of $\F$ consisting
of objects of the form $\FU$ with $U$ an open subset of a Euclidean
space, then by Example~\ref{exa:concrete-sheaves}\ref{ex:Fermat},
$\F'$ is also a concrete site. It seems more natural to relate the
category $\CSh(\F')$ to the category $\Diff$. We will show in Subsection~\ref{sub:why}
in what sense the category $\FDiff$ is better than the category $\CSh(\F')$.
\end{rem}
In the above two examples, note that every object in the concrete
site $\So$ or $\F$ is not just a set, but a topological space, and
every morphism is continuous. More generally, assume that a concrete
site $\mathcal{A}$ is a subcategory of the category $\Top$ of topological
spaces and continuous maps, with covers the open coverings. Then every
concrete sheaf $X$ over $\mathcal{A}$ has a canonical topology,
which is the final topology with respect to all sections $A\ra X$,
i.e., the largest topology on the set $|X|$ making all sections continuous.
This defines a functor $\CSh(\mathcal{A})\ra\Top$. This functor sends
every object in $\mathcal{A}$ to the same topological space. When
$\mathcal{A}=\So$, this topology is called the \emph{D-topology}%
\footnote{\emph{The letter ``D'' in ``D-topology'' refers to ``diffeology'',
not the first-order infinitesimals $D$ introduced in the previous
section. We use the same convention for the terminology ``D-open''
in the following sections.}%
} on diffeological spaces (see~\cite{Ch-Si-Wu14} for detailed discussion),
and when $\mathcal{A}=\F$, this topology is called the \emph{Fermat
topology} on Fermat spaces. Moreover, this functor $\CSh(\mathcal{A})\ra\Top$
has a right adjoint, sending every topological space $Y$ to a concrete
sheaf over $\mathcal{A}$ with sections over an object $A$ in $\mathcal{A}$
the set of all continuous maps $A\ra Y$.

\section{Extending diffeological spaces with infinitesimals}

We use the following notations as in Examples~\ref{exa:concrete-sheaves}\ref{ex:diffeology}
and~\ref{exa:concrete-sheaves}\ref{ex:Fermat} throughout this section:
$\So$ is the Souriau site, $\F$ is the Fermat site, $\Diff$ is
the category of diffeological spaces and smooth maps, and $\FDiff$
is the category of Fermat spaces and Fermat maps.

From Examples~\ref{exa:concrete-sheaves}\ref{ex:diffeology} and~\ref{exa:concrete-sheaves}\ref{ex:Fermat},
we know that both categories $\Diff$ and $\FDiff$ are concrete sheaves
over concrete sites $\So$ and $\F$, respectively. In order to find
relationship between categories of concrete sheaves, we only need
to find ``good'' functors between the two sites. There are already
some candidates for such functors introduced in Section~\ref{sec:BasicFermatReals},
and we will use them to build the adding and the deleting infinitesimal
functors.

\subsection{\label{sub:add}The adding infinitesimal functor $^{\bullet}(\blank)$}

In~\cite[Chapters~7-10]{Gio09}, an attempt of adding infinitesimals
on smooth spaces has been made, by using smooth functions from diffeological
spaces to $\R$. The theory goes well for smooth manifolds, or more
generally for separated diffeological spaces, i.e., diffeological
spaces whose smooth functions to $\R$ separate points. But if we
take the diffeological space to be a $1$-dimensional irrational torus,
then after that procedure of adding infinitesimals, we get a trivial
Fermat space (i.e., a single point), since the $D$-topology on any
irrational torus is indiscrete -- the only open subsets are the empty
set and the whole space. In other words, that procedure of adding
infinitesimals turns an important and highly non-trivial diffeological
space into a trivial Fermat space. In this subsection, we introduce
a new approach to extend diffeological spaces with infinitesimals
to overcome this problem, and still keep all the nice properties as
stated in~\cite[Chapters~7-10]{Gio09} for general diffeological
spaces, instead of sepatated ones.

\medskip{}

We introduce the following functor from diffeological spaces to Fermat
spaces, using Fermat extension of smooth functions:
\begin{prop}
\label{prop:Fermat-extension}The assignment $\So\ra\F$ by 
\[
f:U\ra V\mapsto\Ff:\FU\ra\FV
\]
is a functor between the two sites, and hence induces a functor $^{\bullet}(\blank):\Diff\ra\FDiff$
by 
\[
X=\colim_{U\in\So/X}U\in\Diff\mapsto\FX=\colim_{U\in\So/X}\FU\in\mathcal{\FDiff}.
\]

\end{prop}
Note that although the above two colimits have the same indexing category,
the colimits are taken in different categories. We call the functor
$^{\bullet}(\blank):\Diff\ra\FDiff$ the \emph{adding infinitesimal
functor}. Since $\FDiff$ is a category of concrete sheaves, every
point in the Fermat space $\FX$ can be thought of as a point in $\FU$
for some plot $p:U\ra X$. Two such points in $\FX$ are equal if
and only if they are connected by the Fermat extension of a zig-zag
diagram of plots of $X$, instead of using smooth functions $X\ra\R$.
We will see in next proposition that the adding infinitesimal functor
$^{\bullet}(\blank)$ is different in general from the one introduced
in~\cite[Chapter~9]{Gio09}, although we use the same notation. In
particular, this functor sends $U\in\So$ to $\FU$, which coincides
with the notation introduced in Section~\ref{sec:BasicFermatReals},
since the indexing category $\So/U$ has a terminal object $1_{U}:U\ra U$.
\begin{proof}
This is straightforward. Indeed, this is the left Kan extension (see~\cite[X.3]{Mac03})
of the composite of functors $\So\ra\F\ra\FDiff$ along the inclusion
functor $\So\ra\Diff$.
\end{proof}
Here is the relationship between the underlying sets of $X$ and $\FX$:
\begin{prop}
\label{prop:difference}The adding infinitesimal functor $^{\bullet}(\blank):\Diff\ra\FDiff$
makes every diffeological space a subset of the corresponding Fermat
space.
\end{prop}
In particular, if $X$ is a $1$-dimensional irrational torus, then
$|X|$ is a subset of $|\FX|$, which implies that $\FX$ is not a
trivial Fermat space; see Example~\ref{exa:irrational-torus} for
the final answer of $\FX$. Therefore, the adding infinitesimal functor
$^{\bullet}(\blank)$ is different from the one introduced in~\cite[Chapter~9]{Gio09}.
\begin{proof}
Since the functor $|\blank|:\CSh(\mathcal{A})\ra\Set$ has a right
adjoint for any concrete site $\mathcal{A}$, it preserves colimits,
i.e., for any diffeological space $X$, we have 
\[
|X|=|\colim_{U\in\So/X}U|=\colim_{U\in\So/X}|U|\text{ and }|\FX|=|\colim_{U\in\So/X}\FU|=\colim_{U\in\So/X}|\FU|\text{ in }\Set.
\]
Recall that for any smooth map $f:U\ra V$ between open subsets of
Euclidean spaces, we have the following commutative diagram in $\Set$:
\[
\xymatrix{|U|\ar[d]_{|f|}\ar[r]^{i_{U}} & |\FU|\ar[d]_{|\Ff|}\ar[r]^{ev_{0}} & |U|\ar[d]^{|f|}\\
|V|\ar[r]_{i_{V}} & |\FV|\ar[r]_{ev_{0}} & |V|,
}
\]
and the composites of the two horizontal maps are identities. Therefore,
we have maps $i_{X}:|X|\ra|\FX|$ and $ev_{0}:|\FX|\ra|X|$ such that
$ev_{0}\circ i_{X}=1_{|X|}$. This implies that $i_{X}$ is injective,
and hence $|X|$ is a subset of $|\FX|$.
\end{proof}
Moreover, for any smooth map $f:X\ra Y$ between diffeological spaces,
we have the following commutative diagram in $\Set$: 
\[
\xymatrix{|X|\ar[d]_{|f|}\ar[r]^{i_{X}} & |\FX|\ar[d]_{|\Ff|}\ar[r]^{ev_{0}} & |X|\ar[d]^{|f|}\\
|Y|\ar[r]_{i_{Y}} & |\FY|\ar[r]_{ev_{0}} & |Y|.
}
\]
(Actually this holds in $\Top$, where $X$ and $Y$ are equipped
with the D-topology, and $\FX$ and $\FY$ are equipped with the Fermat
topology. But we will not need this fact in this paper.) In other
words, $|\Ff|$ is always a retract of $|f|$. Therefore, if $|\Ff|$
is injective (resp. surjective or bijective), then so is $|f|$. When
$X$ and $Y$ are open subsets of Euclidean spaces, $\Ff$ coincides
with the notation introduced in Section~\ref{sec:BasicFermatReals}.

\medskip{}

The adding infinitesimal functor behaves nicely with respect to D-open
subsets:
\begin{prop}
\label{prop:D=00003D>Fermat} Let $A$ be a D-open subset of a diffeological
space $X$, equipped with the subset diffeology. Then $\FA$ is a
Fermat open subset of $\FX$.\end{prop}
\begin{proof}
Let $i:A\hookrightarrow X$ be the inclusion map, which induces a
Fermat map $^{\bullet}i:\FA\ra\FX$. Since $A$ is a D-open subset
of $X$, for any plot $p:U\ra X$, $p^{-1}(A)\subseteq U$ is open
and $p|_{p^{-1}(A)}:p^{-1}(A)\ra A$ is a plot of $A$. So we get
a functor $\So/X\ra\So/A$, such that the composite $\So/A\hookrightarrow\So/X\ra\So/A$
is identity. This does not mean that we always have a Fermat map $\FX\ra\FA$,
but from this it follows that $^{\bullet}i$ is injective.

For any Fermat plot $q:B\ra\ext{X}=\colim_{U\in\So/X}\FU$ and any
point $b\in B$, there exist a Fermat open neighborhood $C$ of $b$
in $B$, some plot $r:U\ra X$, and a quasi-standard smooth map $f:C\ra\FU$
such that the following square commutes in $\FDiff$: 
\[
\xymatrix{C~\ar[d]_{f}\ar@{^{(}->}[r] & B\ar[d]^{q}\\
\FU\ar[r]_{\ext{r}} & \FX.
}
\]
Since every quasi-standard smooth map is continuous with respect to
the Fermat topology, it is enough to prove that $(\ext{r})^{-1}(\FA)={^{\bullet}(r^{-1}(A))}$,
which is the statement of next proposition.\end{proof}
\begin{prop}
Let $f:X\ra Y$ be a smooth map between diffeological spaces, and
let $A$ be a D-open subset of $Y$, equipped with the subset diffeology
of $Y$. Also equip $f^{-1}(A)$ with the subset diffeology of $X$.
Then 
\[
(\Ff)^{-1}(\FA)={^{\bullet}(f^{-1}(A))}.
\]
\end{prop}
\begin{proof}
Since $f:X\ra Y$ is smooth and $A$ is D-open in $Y$, $f^{-1}(A)$
is D-open in $X$. From what we have proved in the previous proposition,
we know that the inclusion map $f^{-1}(A)\hookrightarrow X$ induces
an injective map $^{\bullet}(f^{-1}(A))\ra\FX$. So both $(\Ff)^{-1}(\FA)$
and ${^{\bullet}(f^{-1}(A))}$ are subsets of $\FX$.

For any plot $p:U\ra f^{-1}(A)$, we have the following commutative
diagram in $\Diff$: 
\[
\xymatrix{U\ar[r]^{p} & f^{-1}(A)~\ar@{^{(}->}[r]\ar[d]_{f|_{f^{-1}(A)}} & X\ar[d]^{f}\\
 & A~\ar@{^{(}->}[r] & Y,
}
\]
which induces a commutative square in $\FDiff$: 
\[
\xymatrix{\FU\ar[r]\ar[d] & \FX\ar[d]^{\Ff}\\
\FA~\ar@{^{(}->}[r] & \FY.
}
\]
Therefore, $\colim_{U\in\So/f^{-1}(A)}\FU={^{\bullet}(f^{-1}(A))}\subseteq(\Ff)^{-1}(\FA)$.

For the converse inclusion, assume that 
\[
\Ff(x)\in\FA=\colim_{V\in\So/A}\FV
\]
for some $x\in\FX=\colim_{U\in\So/X}\FU$. So there exist plots $p:U\ra X$
and $q:V\ra A$, and points $u\in\FU$ and $v\in\FV$ such that $\Fp:\FU\ra\FX$
sends $u$ to $x$ and $\Fq:\FV\ra\FA$ sends $v$ to $\Ff(x)$. That
is, $\Fq(v)={^{\bullet}(f\circ p)}(u)\in\FY=\colim_{W\in\So/Y}\FW$.

Since $|\blank|:\FDiff\ra\Set$ is faithful and has a right adjoint,
there exist finitely many plots $r_{i}:W_{i}\ra Y$, points $w_{i}\in\FW_{i}$
and zig-zag morphisms in $\So/Y$ connecting $f\circ p$ and $j\circ q$
via these $r_{i}$'s, where $j:A\hookrightarrow Y$, so that $u$
and $v$ are connected via these $w_{i}$'s when applying the adding
infinitesimal functor on the zig-zag. Let us do the following to ``shorten''
the length of the zig-zag:

(1) If we have the following commutative triangle in $\Diff$: 
\[
\xymatrix{W\ar[rr]^{g}\ar[dr]_{r} &  & V\ar[dl]^{q}\\
 & Y,
}
\]
then $r$ can also be viewed as a plot of $A$, so we switch to consider
the pair $(r,w)$ with $w\in\FW$ given (so $\Fg(w)=v$) instead of
$(q,v)$;

(2) If we have the following commutative triangle in $\Diff$: 
\[
\xymatrix{W\ar[dr]_{r} &  & V\ar[dl]^{q}\ar[ll]_{g}\\
 & Y,
}
\]
then $r^{-1}(A)\neq\emptyset$, and the given $w\in\FW$ is actually
in $r^{-1}(A)$. So we switch to consider the pair $(r|_{r^{-1}(A)},w)$
instead of $(q,v)$. In this case, we might need to shrink one $W_{i}$
next to $W$ or $U$ a bit to keep the zig-zag in $\So/Y$, but without
changing the given points $w_{i}$.

After finitely many steps of switching pairs, we know that there exists
an open neighborhood $U'$ of $\st{u}$ in $U$ such that $f(p(U'))\subseteq A$.
Therefore, $(\Ff)^{-1}(\FA)\subseteq{^{\bullet}(f^{-1}(A))}$.
\end{proof}
In the next two results, we are going to connect the D-topology on
a diffeological space $X$ and the Fermat topology on $\ext{X}$.
\begin{prop}
\label{prop:Fermat=00003D>D} Let $X$ be a diffeological space, and
let $A$ be a Fermat open subset of $\FX$. Then $X\cap A$ is a D-open
subset of $X$, and $A={^{\bullet}(X\cap A)}$.\end{prop}
\begin{proof}
Let $p:U\ra X$ be an arbitrary plot. Using the commutative square
\[
\xymatrix{|U|\ar[r]^{i_{U}}\ar[d]_{p} & |\FU|\ar[d]^{\Fp}\\
|X|\ar[r]_{i_{X}} & |\FX|,
}
\]
it is straightforward to check that $p^{-1}(X\cap A)=U\cap(\Fp)^{-1}(A)$,
and hence $X\cap A$ is a D-open subset of $X$. So both $A$ and
$^{\bullet}(X\cap A)$ are Fermat open subsets of $\FX$.

Note that every point in $^{\bullet}(X\cap A)$ can be represented
by $v_{q}\in\FV$, where $q:V\ra X$ is a plot whose image is in $X\cap A$.
Since $V=q^{-1}(X\cap A)=V\cap(\Fq)^{-1}(A)$, we have $v_{q}\in A$.
Hence, $^{\bullet}(X\cap A)\subseteq A$.

On the other hand, assume that $w_{r}\in\FW$ with $r:W\ra X$ a plot
represents a point in $A$, i.e., $\ext{r}(w_{r})\in A$. Since $A$
is Fermat open in $X$, $(\ext{r})^{-1}(A)$ is Fermat open in $\FW$,
which implies that $r(\st{w_{r}})\in X\cap A$, and hence $\ext{r}(w_{r})\in{^{\bullet}(X\cap A)}$.
Therefore, $A\subseteq{^{\bullet}(X\cap A)}$.

As a result, we have $A={^{\bullet}(X\cap A)}$.
\end{proof}
In conclusion, we have:
\begin{thm}
Let $X$ be a diffeological space. Then there is a bijection between
the D-open subsets of $X$ and the Fermat open subsets of $\FX$.\end{thm}
\begin{proof}
The maps between these sets are given by sending a D-open subset $A$
of $X$ to $\FA$, and by sending a Fermat open subset $B$ of $\FX$
to $X\cap B$, respectively. To prove that these maps are inverse
to each other, by Propositions~\ref{prop:D=00003D>Fermat} and~\ref{prop:Fermat=00003D>D},
we are left to show that $X\cap\FA=A$. Assume that $u_{p}\in U$
with plot $p:U\ra X$ and $v_{q}\in\FV$ with plot $q:V\ra X$ whose
image is in $A$ represent the same element in $\FX$. By using $ev_{0}$,
it is clear that $\st{v_{q}}$ and $u$ represent the same element
in $X$, and the former actually represents an element in $A$. Hence,
$X\cap\FA\subseteq A$. The converse inclusion is clear.
\end{proof}
The next two results are easy applications:
\begin{cor}
\label{cor:properties-of-add-infinitesimal}Let $X$ be a diffeological
space, and let $\{A_{i}\}_{i\in I}$ be a set of D-open subsets of
$X$. Then we have 
\[
^{\bullet}(A_{1}\cap A_{2})=\FA_{1}\cap\FA_{2},
\]
\[
^{\bullet}(\bigcup_{i\in I}A_{i})=\bigcup_{i\in I}\FA_{i},
\]
and 
\[
^{\bullet}(\text{int}(X\setminus A_{1}))=\text{int}(\FX\setminus\FA_{1}),
\]
 where $\text{int}$ denotes the interior. \end{cor}
\begin{prop}
\label{prop:open-map}Let $f:X\ra Y$ be a smooth map between diffeological
spaces, which is an open map with respect to the D-topology. Let $A$
be a D-open subset of $X$. Then $^{\bullet}(f(A))=(\Ff)(\FA)$, where
$f(A)$ and $A$ are equipped with the subset diffeology of $Y$ and
$X$, respectively.\end{prop}
\begin{proof}
Since $A$ is a D-open subset of $X$ and $f:X\ra Y$ is an open map,
$f(A)$ is a D-open subset of $Y$. Then $(\Ff)(\FA)\subseteq{^{\bullet}(f(A))}$
follows from applying the functor $^{\bullet}(\blank)$ to the commutative
square 
\[
\xymatrix{A\ \ar@{^{(}->}[r]\ar[d]_{f|_{A}} & X\ar[d]^{f}\\
f(A)\ \ar@{^{(}->}[r] & Y
}
\]
together with Proposition~\ref{prop:D=00003D>Fermat}. Since $f(A)\subseteq(\Ff)(\FA)$,
we have $^{\bullet}(f(A))\subseteq(\Ff)(\FA)$ by Proposition~\ref{prop:Fermat=00003D>D}.
Therefore, $^{\bullet}(f(A))=(\Ff)(\FA)$.
\end{proof}
For quotient spaces, we have:
\begin{prop}
If $Y$ is a quotient space of a diffeological space $X$, then $\FY$
is a quotient space of the Fermat space $\FX$.\end{prop}
\begin{proof}
Since $Y$ is a quotient space of $X$, every plot of $Y$ locally
factors through a plot of $X$, which implies that the quotient map
$X\ra Y$ induces a surjective map $|\FX|\ra|\FY|$, and moreover,
every Fermat plot of $\FY$ also locally factors through a Fermat
plot of $\FX$.
\end{proof}
The adding infinitesimal functor preserves finite products:
\begin{prop}
For any diffeological spaces $X_{1}$ and $X_{2}$, we have a natural
isomorphism $^{\bullet}(X_{1}\times X_{2})\cong\FX_{1}\times\FX_{2}$
in $\FDiff$.\end{prop}
\begin{proof}
Note that 
\[
^{\bullet}(X_{1}\times X_{2})=\colim_{U\in\So/(X_{1}\times X_{2})}\FU,
\]
and 
\[
\begin{split}\FX_{1}\times\FX_{2} & =(\colim_{V\in\So/X_{1}}\FV)\times(\colim_{W\in\So/X_{2}}\FW)\\
 & =\colim_{V\in\So/X_{1}}(\FV\times\colim_{W\in\So/X_{2}}\FW)\\
 & =\colim_{(V\in\So/X_{1})\times(W\in\So/X_{2})}(\FV\times\FW)\\
 & \cong\colim_{(V\in\So/X_{1})\times(W\in\So/X_{2})}{^{\bullet}(V\times W)}
\end{split}
\]
where the second and the third equalities follow from Cartesian closedness
of $\FDiff$, and the last isomorphism in $\FDiff$ is~\cite[Theorem~19]{Gio11}.

We can define a functor $(\So/X_{1})\times(\So/X_{2})\ra\So/(X_{1}\times X_{2})$
sending $(f,g):(q:V\ra X_{1},r:W\ra X_{2})\ra(q':V'\ra X_{1},r':W'\ra X_{2})$
to $(f\times g):(q\times r:V\times W\ra X_{1}\times X_{2})\ra(q'\times r':V'\times W'\ra X_{1}\times X_{2})$.
It is straightforward to check that this functor is final (\cite[Section~IX.3]{Mac03}),
and hence $\FX_{1}\times\FX_{2}\ra{}^{\bullet}(X_{1}\times X_{2})$
is an isomorphism%
\footnote{The inverse of this isomorphism is induced by the projections $\pi_{i}:X_{1}\times X_{2}\ra X_{i}$
for $i=1,2$.%
} in $\FDiff$ (\cite[Theorem~IX.3.1]{Mac03}).

The naturality means that if $f_{1}:X_{1}\ra X_{1}'$ and $f_{2}:X_{2}\ra X_{2}'$
are smooth maps between diffeological spaces, then we have a commutative
square in $\FDiff$: 
\[
\xymatrix{\FX_{1}\times\FX_{2}\ar[rr]^{\Ff_{1}\times\Ff_{2}}\ar[d] &  & \FX_{1}'\times\FX_{2}'\ar[d]\\
{^{\bullet}(X_{1}\times X_{2})}\ar[rr]_{{^{\bullet}(f_{1}\times f_{2})}} &  & {^{\bullet}(X_{1}'\times X_{2}')}.
}
\]
This follows easily from the canonical map $\FX_{1}\times\FX_{2}\ra{}^{\bullet}(X_{1}\times X_{2})$
described above.\end{proof}
\begin{rem}
\label{rem:product} More generally, we have the following result
by a similar proof. Let $\mathcal{A}$ and $\mathcal{B}$ be concrete
sites with finite products, and let $F:\mathcal{A}\ra\mathcal{B}$
be a natural finite-product-preserving functor. Then the induced functor
$F:\CSh(\mathcal{A})\ra\CSh(\mathcal{B})$ defined by $X=\colim_{A\in\mathcal{A}/X}A\mapsto F(X):=\colim_{A\in\mathcal{A}/X}F(A)$
naturally preserves finite products. This result will be used in Proposition~\ref{prop:productF}.
\end{rem}
Now we discuss function spaces. Let $X$ and $Y$ be diffeological
spaces. Since the category $\Diff$ of diffeological spaces is Cartesian
closed, $\Diff(X,Y)$ is also a diffeological space, with the natural
diffeology (called the \emph{functional diffeology}) consisting of
all maps $U\ra\Diff(X,Y)$ such that the corresponding adjoint maps
$U\times X\ra Y$ are smooth. So $^{\bullet}(\Diff(X,Y))=\colim_{U\in\So/\Diff(X,Y)}\FU$.
On the other hand, we can apply the adding infinitesimal functor to
the adjoint maps $U\times X\ra Y$ to get $\FU\times\FX\cong{^{\bullet}(U\times X)}\ra\FY$.
Since the category $\FDiff$ of Fermat spaces is Cartesian closed,
we can take the adjoint back and get Fermat maps $\FU\ra\FDiff(\FX,\FY)$.
It is easy to check that we get a Fermat map $i:{^{\bullet}(\Diff(X,Y))}\ra\FDiff(\FX,\FY)$.
Moreover, the composite 
\[
\xymatrix{|\Diff(X,Y)|\ar[rr]^{i_{\Diff(X,Y)}} &  & |{^{\bullet}(\Diff(X,Y))}|\ar[r]^{|i|} & |\FDiff(\FX,\FY)|}
\]
exactly sends a smooth map $f$ to its Fermat extension $\Ff$.

In general, one cannot expect the Fermat map $i:{^{\bullet}(\Diff(X,Y))}\ra\FDiff(\FX,\FY)$
to be an isomorphism in $\FDiff$. For example, when $X=Y=\R$, $^{\bullet}(\Diff(\R,\R))$
consists of $\Ff(u,\blank):\FR\ra\FR$, where $f:U\times\R\ra\R$
is a smooth map with $U$ some open subset of a Euclidean space, and
$u\in\FU$ is some fixed point; $\FDiff(\FR,\FR)$ is the set of all
Fermat maps $\FR\ra\FR$; the map $i:{^{\bullet}(\Diff(\R,\R))}\ra\FDiff(\FR,\FR)$
is the inclusion map, which is hence not an isomorphism in $\FDiff$.

On the other hand, we will show in next subsection that both $^{\bullet}(\Diff(X,Y))$
and $\FDiff(\FX,\FY)$ have the same ``underlying diffeological space''.

\subsection{\label{sub:delete}The deleting infinitesimal functor $_{\bullet}(\blank)$}

In this subsection, we introduce a functor $\FDiff\ra\Diff$ which
deletes all infinitesimal points. This is the left inverse of the
adding infinitesimal functor $^{\bullet}(-)$ introduced in the previous
subsection.

\medskip{}

Recall that we write $\F'$ for the full subsite of the Fermat site
$\F$, consisting of objects of the form $\ext{U}$ for $U$ an open
subset of $\R^{n}$ for all $n\in\N$. We first observe that
\begin{prop}
\label{prop:functordel} $\F'\ra\So$ defined by 
\[
f:\ext{U}\ra\ext{V}\,\,\mapsto\,\,\Lf:U\ra V,
\]
with $\Lf(u)=ev_{0}\circ f\circ i_{U}(u)=\st{(f(u))}$ is a functor.\end{prop}
\begin{proof}
Note that $f$ is quasi-standard smooth, i.e., for every $a\in\FU$,
there exist an open neighborhood $U'$ of $\st{a}$ in $U$, an open
subset $U''$ of a Euclidean space, a fixed point $b\in\FU''$, and
a smooth map $g:U''\times U'\ra\R^{n}$ with $n=\dim(V)$, such that
for any $x\in\FU'$, $f(x)=\Fg(b,x)$. Hence, for any $u\in U'$,
$\Lf(u)=\st{(f(u))}=\st{(\Fg(b,u))}=g(\st{b},u)$. Therefore, $\Lf$
is a smooth map.

Clearly $_{\bullet}(1_{\FU})=1_{U}$.

Let $f:\FU\ra\FV$ and $g:\FV\ra\FW$ be quasi-standard smooth maps.
Then for any $u\in U$ 
\[
\begin{split}\Lg(\Lf(u)) & =\Lg(\st{(f(u))})\\
 & =\st{(g(\st{(f(u))}))}\\
 & =\st{(g(f(u)))}\\
 & ={_{\bullet}(g\circ f)(u)},
\end{split}
\]
where the third equality follows from Taylor's expansion of the local
expression of $g$ as a Fermat extension of a smooth function. Therefore,
$\Lg\circ\Lf={_{\bullet}(g\circ f)}$.
\end{proof}
Hence, we get a functor from Fermat spaces to diffeological spaces:
\begin{prop}
\label{prop:def-delete}$_{\bullet}(\blank):\FDiff\ra\Diff$ defined
by 
\[
X\mapsto\LX=\colim_{\ext{U}\in\F'/X}U
\]
is a functor.
\end{prop}
We call this functor the \emph{deleting infinitesimal functor}; see
next proposition for explanation.
\begin{proof}
This is clear from Proposition~\ref{prop:functordel}. Indeed, this
is the left Kan extension of the composite of functors $\F'\ra\So\ra\Diff$
along the inclusion functor $\F'\ra\FDiff$.
\end{proof}
It is easy to check that the composite $\So\ra\F'\ra\So$ is identity,
where the first functor is introduced in Proposition~\ref{prop:Fermat-extension},
and the second one is given by Proposition~\ref{prop:functordel}.
This property can be extended to the corresponding concrete sheaf
categories:
\begin{prop}
\label{prop:delete} The composite 
\[
\xymatrix{\Diff\ar[r]^{^{\bullet}(\blank)} & \FDiff\ar[r]^{_{\bullet}(\blank)} & \Diff}
\]
is the identity functor.
\end{prop}
In other words, the deleting infinitesimal functor is the left inverse
of the adding infinitesimal functor.
\begin{proof}
For any diffeological space $X$, we prove below that $_{\bullet}(\FX)=X$.
From the proof, it is clear that the composite of these two functors
acts as identity on morphisms.

Recall that 
\[
X=\colim_{U\in\So/X}U,\,\,\FX=\colim_{U\in\So/X}\FU,\text{ and }\,{_{\bullet}(\FX)}=\colim_{\ext{U}\in\F'/\FX}U.
\]
We define a functor $\So/X\ra\F'/\FX$ by 
\[
\xymatrix{U\ar[rr]^{f}\ar[dr]_{p} &  & V\ar[dl]^{q}\\
 & X
}
\mapsto\xymatrix{\FU\ar[rr]^{\Ff}\ar[dr]_{\Fp} &  & \FV\ar[dl]^{\Fq}\\
 & \FX.
}
\]
It is straightforward to check that $_{\bullet}(\Ff)=f:U\ra V$, and
hence we get a natural smooth map $X\ra{_{\bullet}(\FX)}$.

On the other hand, for any Fermat plot $p:\FU\ra\FX$, write $\bar{p}:U\ra X$
for the composite 
\[
\xymatrix{U\ar[r]^{i_{U}} & \FU\ar[r]^{p} & \FX\ar[r]^{ev_{0}} & X.}
\]
By a similar proof as Proposition~\ref{prop:functordel}, one can
check that $\F'/\FX\ra\So/X$ defined by 
\[
\xymatrix{\FU\ar[rr]^{f}\ar[dr]_{p} &  & \FV\ar[dl]^{q}\\
 & \FX
}
\mapsto\xymatrix{U\ar[rr]^{\Lf}\ar[dr]_{\bar{p}} &  & V\ar[dl]^{\bar{q}}\\
 & X
}
\]
is a well-defined functor, and hence we get another natural smooth
map $_{\bullet}(\FX)\ra X$.

Although $p$ and $\ext{\bar{p}}$ can be different, it is straightforward
to check that the two composites 
\[
\xymatrix{U\ar[r]^{i_{U}} & \FU\ar@<2pt>[r]^{p}\ar@<-2pt>[r]_{^{\bullet}\bar{p}} & \FX\ar[r]^{ev_{0}} & X}
\]
are the same, and hence the two maps $X\ra{_{\bullet}(\FX)}$ and
$_{\bullet}(\FX)\ra X$ are inverse to each other.
\end{proof}
By a similar method, one can show that if $X$ and $Y$ are diffeological
spaces, and $f:\FX\ra\FY$ is a Fermat map, then $\Lf:{_{\bullet}(\FX)}\ra{_{\bullet}(\FY)}$
after natural diffeomorphisms as constructed in the proof of Proposition~\ref{prop:delete}
corresponds to $\bar{f}:X\ra Y$, i.e., the composite 
\[
\xymatrix{X\ar[r]^{i_{X}} & \FX\ar[r]^{f} & \FY\ar[r]^{ev_{0}} & Y.}
\]
In particular, from a commutative triangle in $\FDiff$: 
\[
\xymatrix{\FX\ar[rr]^{f}\ar[dr]_{h} &  & \FY\ar[dl]^{g}\\
 & \FZ,
}
\]
we get a commutative triangle in $\Diff$: 
\[
\xymatrix{X\ar[rr]^{\bar{f}}\ar[dr]_{\bar{h}} &  & Y\ar[dl]^{\bar{g}}\\
 & Z,
}
\]
where $X,Y,Z$ are diffeological spaces, and $f,g,h$ are Fermat maps.

As an easy application, we have:
\begin{cor}
Let $X$ and $Y$ be diffeological spaces. Then the Fermat map $i:{^{\bullet}(\Diff(X,Y))}\ra\FDiff(\FX,\FY)$
introduced at the end of last subsection induces a diffeomorphism
$\Diff(X,Y)\cong{_{\bullet}(\FDiff(\FX,\FY))}$. \end{cor}
\begin{proof}
A morphism 
\[
\xymatrix{\FU\ar[rr]^{f}\ar[dr]_{\tilde{p}} &  & \FV\ar[dl]^{\tilde{q}}\\
 & \FDiff(\FX,\FY)
}
\]
in $\F'/\FDiff(\FX,\FY)$ is equivalent to a commutative triangle
\[
\xymatrix{\FU\times\FX\ar[rr]^{f\times1_{\FX}}\ar[dr]_{p} &  & \FV\times\FX\ar[dl]^{q}\\
 & \FY
}
\]
in $\FDiff$. By the observation above this corollary, we get a commutative
triangle 
\[
\xymatrix{U\times X\ar[rr]^{\bar{f}\times1_{X}}\ar[dr]_{\bar{p}} &  & V\times X\ar[dl]^{\bar{q}}\\
 & Y
}
\]
in $\Diff$. By Cartesian closedness of $\Diff$, we get a morphism
\[
\xymatrix{U\ar[rr]^{\bar{f}=\Lf}\ar[dr] &  & V\ar[dl]\\
 & \Diff(X,Y)
}
\]
in $\So/\Diff(X,Y)$. Hence, we get a smooth map ${_{\bullet}(\FDiff(\FX,\FY))}\ra\Diff(X,Y)$,
and one can check easily that this is the inverse of ${_{\bullet}i}$.
\end{proof}
On the other hand, we will show in next subsection that $^{\bullet}(\LY)$
is not isomorphic to $Y$ in $\FDiff$ for a general Fermat space
$Y$.

\bigskip{}

Here is another way to think of the deleting infinitesimal functor.
The inclusion of the concrete sites $\F'\hookrightarrow\F$ gives
rise to the restriction functor $\FDiff\ra\CSh(\F')$, and Proposition~\ref{prop:functordel}
induces a functor $\CSh(\F')\ra\Diff$. One can check that the functor
$\F'\ra\So$ naturally preserves finite products, and as a result
of Remark~\ref{rem:product} we have:
\begin{prop}
\label{prop:productF} For any Fermat spaces $X$ and $Y$, $_{\bullet}(X\times Y)$
is naturally diffeomorphic to $\LX\times\LY$.
\end{prop}
For application of the adding and the deleting infinitesimal functors
to integrals, see the subsection ``Standard and infinitesimal parts
of an integral'' in \cite[Section�7]{Gi-Wu15}.

\subsection{\label{sub:why}Why we choose $\F$ to be the Fermat site}

In Example~\ref{exa:concrete-sheaves}\ref{ex:Fermat}, we defined
$\F$ with objects all subsets of $\FR^{n}$ for all $n\in\N$, morphisms
all quasi-standard smooth maps between them, and covers the Fermat
open coverings to be the Fermat site. We have explained in Section~\ref{sec:BasicFermatReals}
that instead of taking morphisms to consist of only Fermat extension
of smooth maps, we get a concrete site. In order to relate Fermat
spaces with diffeological spaces, there is another natural choice
-- we can take $\F'$ to be the full subcategory of $\F$ consisting
of objects of the form $\FU$ with $U$ an open subset of a Euclidean
space. Then by Example~\ref{exa:concrete-sheaves}\ref{ex:Fermat},
$\F'$ is also a concrete site. In this subsection, we explain in
what sense the category $\FDiff$ is better than the category $\CSh(\F')$.

\medskip{}

One naive reason is, we want to develop geometry of Fermat spaces
and diffeological spaces using general spaces like $D_{\infty}=\{x\in\FR\mid\st{x}=0\}$,
$D=\{x\in\FR\mid x^{2}=0\}$ and $D_{\geq0}=\{x\in D\mid x\geq0\}$,
but none of them are of the form $\FU$ for some open subset $U$
of a Euclidean space. On the other hand, these spaces are in the site
$\F$. However, since $\CSh(\F')$ is the category of concrete sheaves
over the concrete site $\F'$, we can think of $D_{\infty}$, $D$
and $D_{\geq0}$ as subspaces of $\FR$ in $\CSh(\F')$.

We can also define the adding infinitesimal functor $\Diff\ra\CSh(\F')$
as in Subsection~\ref{sub:add}, denoted by $X\mapsto\FX'$%
\footnote{This meaning of the notation $\FX'$ is valid only in this subsection.%
}. The actual reason is, the Fermat space $\FX\in\FDiff$ still keeps
record of the diffeological information of $X$, but $\FX'\in\CSh(\F')$
does not, for the following explanations:
\begin{prop}
\label{prop:smooth} Let $X$ be a diffeological space, and write
$(X\leq\FX)$ for the Fermat subspace of $\FX$ via the inclusion
map $i_{X}:X\ra\FX$. Then for any open subset $U$ of a Euclidean
space, $|\FDiff(U,(X\leq\FX))|=|\Diff(U,X)|$.
\end{prop}
In other words, the Fermat space $\FX$ still remembers $X$ as a
diffeological space.
\begin{proof}
We only need to prove the two inclusions since both $|\FDiff(U,(X\leq\FX))|$
and $|\Diff(U,X)|$ are subsets of $\Set(|U|,|X|)$.

For any plot $p:U\ra X$, we have a commutative square in $\Set$:
\[
\xymatrix{|U|\ar[d]_{|p|}\ar[r]^{i_{U}} & |\FU|\ar[d]^{|\Fp|}\\
|(X\leq\FX)|\ar[r]_{i_{X}} & |\FX|,
}
\]
Since $i_{U}$ is quasi-standard smooth and $\Fp$ is a Fermat map,
$p\in\FDiff(U,(X\leq\FX))$. So we have $|\Diff(U,X)|\subseteq|\FDiff(U,(X\leq\FX))|$.

For any Fermat plot $q:U\ra(X\leq\FX)$, since $\FX=\colim_{V\in\So/X}\FV$,
for every $u\in U$, there exist a Fermat open neighborhood $U'$
of $u$ in $U$, some plot $r:V\ra X$ and a quasi-standard smooth
map $f:U'\ra\FV$ such that $i_{X}\circ q|_{U'}=\ext{r}\circ f$.
Hence, we have the following commutative diagram in $\Set$: 
\[
\xymatrix{|U'|\ar[r]^{|f|}\ar[d]_{|q|_{U'}|} & |\FV|\ar[r]^{ev_{0}}\ar[d]^{|\ext{r}|} & |V|\ar[d]^{|r|}\\
|(X\leq\FX)|\ar[r]_{i_{X}} & |\FX|\ar[r]_{ev_{0}} & |(X\leq\FX)|.
}
\]
Note that the composite of the bottom horizontal maps is $1_{|(X\leq\FX)|}$,
and the composite of the upper horizontal maps is smooth. So $q\in\Diff(U,X)$,
i.e., $|\FDiff(U,(X\leq\FX))|\subseteq|\Diff(U,X)|$.
\end{proof}
As a corollary, we have:
\begin{thm}
\label{thm:Diff->Fermat}The assignment $\Diff\ra\FDiff$ defined
by 
\[
f:X\ra Y\,\,\mapsto\,\, f:(X\leq\FX)\ra(Y\leq\FY)
\]
is a functor, which makes $\Diff$ a full subcategory of $\FDiff$.
\end{thm}
Therefore, every diffeological space is canonically a Fermat space
with the same underlying set, and every smooth map between diffeological
spaces is canonically a Fermat map between the corresponding Fermat
spaces.

Proposition~\ref{prop:smooth} also implies the following:
\begin{example}
Let $U$ and $V$ be open subsets of Euclidean spaces. Then $|\F(U,V)|=|\So(U,V)|$,
i.e., a map $U\ra V$ is quasi-standard smooth if and only if it is
smooth.\end{example}
\begin{rem}
\label{rem:forgetful-Fermat->Diff}In fact, every Fermat space is
automatically a diffeological space in the following way. Let $Y$
be a Fermat space. For any open subset $U$ of a Euclidean space,
we define $U\ra Y$ to be a plot if it is in $\FDiff(U,Y)$. In this
way, $Y$ is a diffeological space. Indeed, this defines a forgetful
functor $\FDiff\ra\Diff$, which has a right inverse given by the
functor defined in Theorem~\ref{thm:Diff->Fermat}.
\end{rem}
By Proposition~\ref{prop:smooth}, if $U$ is an open subset of some
Euclidean space with $\dim(U)>0$, then the Fermat space $(U\leq\FU)$
is not discrete in $\FDiff$.

Now we show that the object $(U\leq\FU')$ in $\CSh(\F')$ is discrete.
Let $p:\FV\ra U$ be a section, i.e., the composite 
\[
\xymatrix{\FV\ar[r]^{p} & U\ar[r]^{i_{U}} & \FU'}
\]
is a quasi-standard smooth map. So for every $v\in\FV$, there exist
an open connected neighborhood $\tilde{V}$ of $\st{v}$ in $V$,
an open subset $W$ of some Euclidean space, a fixed point $w\in\FW$
and a smooth map $f:W\times\tilde{V}\ra\R^{n}$ with $n=\dim(U)$,
such that $i_{U}(p(x))=\Ff(w,x)$ for every $x\in\FtV$. So 
\[
\begin{split}p(x) & =ev_{0}(i_{U}(p(x)))\\
 & =ev_{0}(\Ff(w,x))\\
 & =f(\st{w},\st{x})\in U
\end{split}
\]
for every $x\in\FtV$. In other words, $\Ff(w,x)=f(\st{w},\st{x})\in U$
for the fixed $w\in\FW$ and every $x\in\FtV$. Therefore, by Taylor's
expansion of $f$, $\frac{\partial f}{\partial x_{i}}(\st{w},\st{x})=0$
for every variable $x_{i}$, which implies that $f(\st{w},\st{x})$
is a constant independent of $\st{x}\in\tilde{V}$ since $\tilde{V}$
is connected. That is, every section of $(U\leq\FU')$ is locally
constant, so $(U\leq\FU')$ is discrete in $\CSh(\F')$.

We summarize the above discussion as the following proposition:
\begin{prop}
\label{prop:localconst} Let $U$ and $V$ be open subsets of Euclidean
spaces. Then the set $|\FDiff(\FV,(U\leq\FU))|$ consists of only
locally constant maps.
\end{prop}
Here are some easy corollaries. Note that if $U$ is an open subset
of a Euclidean space of positive dimension, then as a representing
concrete sheaf in $\FDiff$, it is exactly $(U\leq\FU)$. (Indeed,
by definition, for any object $A\subseteq\FR^{n}$ in $\F$, as a
representing concrete sheaf in $\FDiff$, it is exactly $(A\leq\FR^{n})$.)
Therefore, by Proposition~\ref{prop:localconst}, the map $ev_{0}:\FU\ra U$
is not quasi-standard smooth. We also have:
\begin{cor}
Let $U$ be an open subset of some Euclidean space. Then $_{\bullet}(U\leq\FU)$
is the set $U$ with the discrete diffeology.
\end{cor}
Therefore, $^{\bullet}(\LY)$ is not necessarily isomorphic to $Y$
in $\FDiff$ for a general Fermat space $Y$.

\subsection{\label{sub:Calculations}Calculations}

In this subsection, we will do a few calculations for $\FX$ and $\LY$
for diffeological space $X$ and Fermat space $Y$.

\medskip{}

Here is the general situation we will meet frequently, in both $\Diff$
and $\FDiff$:
\begin{thm}
\label{thm:general} Let $\mathcal{A}$ be a concrete site, let $\mathcal{I}$
be a small category, let $\mathcal{J}$ be a subcategory of $\mathcal{I}$
with the inclusion $i:\mathcal{J}\ra\mathcal{I}$, and let $F:\mathcal{I}\ra\CSh(\mathcal{A})$
be a functor. Then the natural map $\colim_{\mathcal{J}}(F\circ i)\ra\colim_{\mathcal{I}}F$
is an isomorphism in $\CSh(\mathcal{A})$ if the following conditions
hold:

(1) for any object $i$ in $\mathcal{I}$ and for any section $c:A\ra F(i)$,
there exists a cover $\{c_{\lambda}:A_{\lambda}\ra A\}_{\lambda\in\Lambda}$
of $A$ such that for each $\lambda$, there exist an object $j$
in $\mathcal{J}$ and a section $d_{\lambda}:A_{\lambda}\ra F(j)$
making the following diagram commutative: 
\[
\xymatrix{A_{\lambda}\ar[r]^{c\circ c_{\lambda}}\ar[d]_{d_{\lambda}} & F(i)\ar[r] & \colim_{\mathcal{I}}F\\
F(j)\ar[rr] &  & \colim_{\mathcal{J}}(F\circ i);\ar[u]
}
\]

(2) it induces an injective set map $|\colim_{\mathcal{J}}(F\circ i)|\ra|\colim_{\mathcal{I}}F|$.\end{thm}
\begin{proof}
Recall that a colimit in a category of concrete sheaves is the colifting
of the corresponding colimit in $\Set$. Condition (1) means that
the map $|\colim_{\mathcal{J}}(F\circ i)|\ra|\colim_{\mathcal{I}}F|$
is surjective, so together with Condition (2), this map is a bijection.
Then use Condition (1) again, it is easy to see that the inverse map
$\colim_{\mathcal{I}}F\ra\colim_{\mathcal{J}}(F\circ i)$ is also
a morphism in $\CSh(\mathcal{A})$.
\end{proof}
The hard part of applying this theorem is to check Condition (2).
We will make it more explicit in the following cases:

\subsubsection{Calculations of $\FX$}

In Subsection~\ref{sub:add}, we defined $\FX=\colim_{U\in\So/X}\FU$
for every diffeological space $X$. One can use this definition to
show that if $X$ is a discrete diffeological space, then $\FX$ is
a discrete Fermat space with $|\FX|=|X|$. But in general, the plot
category $\So/X$ is huge. We need to find a more efficient way to
calculate $\FX$.

In many examples, the diffeological space is given as a colimit of
Euclidean spaces over a small subcategory of its plot category. The
following proposition tells us when we can use this colimit to calculate
the corresponding Fermat space:
\begin{prop}
\label{prop:colimit} Let $X$ be a diffeological space, and let $\mathcal{B}$
be a subcategory of the plot category $\So/X$. Assume $X=\colim_{U\in\mathcal{B}}U$.
Then the natural Fermat map $\colim_{U\in\mathcal{B}}\FU\ra\FX$ is
surjective. If it is also injective, then it is an isomorphism in
$\FDiff$.\end{prop}
\begin{proof}
This is an easy corollary of Theorem~\ref{thm:general}.
\end{proof}
To apply this proposition to calculate $\FX$, the key part is to
check the injectivity of the natural map $\colim_{U\in\mathcal{B}}\FU\ra\FX$.
Injectivity is equivalent to the condition that for any plots $p:U\ra X$
and $q:U'\ra X$ in $\mathcal{B}$, and any points $u\in\FU$ and
$u'\in\FU'$, if there exist plots $V_{1}\ra X,\cdots,V_{n}\ra X$,
points $v_{1}\in\FV_{1},\ldots,v_{n}\in\FV_{n}$, and zig-zag morphisms
among these plots together with $p$ and $q$ in $\So/X$ such that
the Fermat extension of the zig-zag connects $u$ and $u'$ via these
$v_{i}$'s, (by applying $ev_{0}$, this implies that $\st{u}$ and
$\st{(u')}$ represent the same point in $X$), then there exist plots
$U_{1}\ra X,\cdots,U_{m}\ra X$ for some $m\in\N$, points $u_{1}\in\FU_{1},\ldots,u_{m}\in\FU_{m}$,
and zig-zag morphisms among these plots together with $p$ and $q$
in $\mathcal{B}$ such that the Fermat extension of this zig-zag connects
$u$ and $u'$ via these $u_{j}$'s. We will use this description
to calculate the following examples, and from these examples, we abstract
some general results.
\begin{example}
\label{exa:calculation}\ 
\begin{enumerate}[leftmargin=*,label=(\roman*),align=left]
\item \label{enu:calculation-manifold}Let $M$ be a smooth manifold, and
let $\{(U_{i},\phi_{i})\}_{i\in I}$ be a smooth atlas. Then we can
construct a category $\mathcal{I}$ with objects finite subsets of
$I$ and morphisms inclusion maps. There is a canonical functor $\mathcal{I}^{op}\ra\Diff$
sending a finite subset $\{i_{1},\ldots,i_{n}\}\subseteq I$ to $U_{i_{1}}\cap\cdots\cap U_{i_{n}}$,
and sending the inclusion map to the corresponding inclusion map.
It is easy to see that $M$ is the colimit of this functor, so we
write $M=\colim_{i\in\mathcal{I}^{op}}U_{i}$. One can also check
that the injectivity of Proposition~\ref{prop:colimit} holds by
the definition of a smooth atlas on a smooth manifold, and hence $\FM=\colim_{i\in\mathcal{I}^{op}}\FU_{i}$.
\item Let $X$ be the pushout of 
\[
\xymatrix{\R & \R^{0}\ar[l]_{0}\ar[r]^{0} & \R}
\]
in $\Diff$, i.e., $X$ is two real lines glued at the origin. One
can show that the injectivity of Proposition~\ref{prop:colimit}
holds, and hence $\FX$ is the pushout of 
\[
\xymatrix{\FR & \FR^{0}\ar[l]_{0}\ar[r]^{0} & \FR}
\]
in $\FDiff$, i.e., $\FX$ is two Fermat reals glued at the origin.
\item Let $V$ be a fine diffeological vector space (see \cite[Chapter~3]{Igl13}),
and let $\mathcal{I}$ be the poset with objects finite dimensional
linear subspaces of $V$ and morphisms inclusions. Then by \cite[3.8]{Igl13},
it is easy to see that $V=\colim_{W\in\mathcal{I}}W$ and that the
injectivity of Proposition~\ref{prop:colimit} holds, which implies
that $\ext{V}=\colim_{W\in\mathcal{I}}\ext{W}$.
\end{enumerate}
\end{example}
Here is a general result from these three examples:
\begin{prop}
Let $\mathcal{B}$ be a subcategory of the plot category $\So/X$
over a diffeological space $X$. Assume that every object $U\ra X$
in $\mathcal{B}$ is an injective map such that the pullback diffeology
on $U$ coincides with the standard diffeology, and for any objects
$p:U\ra X$ and $q:V\ra X$ with $p(U)\cap q(V)\neq\emptyset$, there
exist an object $r:W\ra X$ with $r(W)=p(U)\cap q(V)$ and morphisms
$r\ra p$ and $r\ra q$ in $\mathcal{B}$. Moreover, if $X=\colim_{U\in\mathcal{B}}U$,
then $\FX=\colim_{U\in\mathcal{B}}\FU$.\end{prop}
\begin{proof}
By Proposition~\ref{prop:colimit}, we are left to check the injectivity
of the natural map $\colim_{U\in\mathcal{B}}\FU\ra\FX$. We split
zig-zag diagrams into two kinds of pieces, and study them separately:

(1) Assume that we have a commutative triangle in $\Diff$: 
\[
\xymatrix{U\ar[rr]^{f}\ar[dr]_{p} &  & V\ar[dl]^{q}\\
 & X
}
\]
with $p$ an object in $\mathcal{B}$, $q$ a plot and $f$ a smooth
map, and $u\in\FU$ and $v\in\FV$ are fixed points such that $\Ff(u)=v$.
Since $X=\colim_{U\in\mathcal{B}}U$, there exist an open neighborhood
$V'$ of $\st{v}$ in $V$, an object $r:U'\ra X$ in $\mathcal{B}$
and a smooth map $g:V'\ra U'$ such that $r\circ g=q|_{V'}$. Therefore,
$p(U)\cap r(U')\neq\emptyset$. By the assumption of the proposition,
there exist an object $s:W\ra X$ with $s(W)=p(U)\cap r(U')$ and
morphisms $s\ra p$ and $s\ra r$ in $\mathcal{B}$. So eventually
we have the following commutative diagram in $\Diff$: 
\[
\xymatrix{ & W\ar[drrr]\ar[dl]\\
U\ar[drrr]_{f} & \ f^{-1}(V')\ar@{_{(}->}[l]\ar[rr]^{f|_{f^{-1}(V')}}\ar[u] &  & V'\ar[r]_{g}\ar@{^{(}->}[d] & U'\\
 &  &  & V.
}
\]

(2) Assume that we have a commutative diagram in $\Diff$: 
\[
\xymatrix{V_{1}\ar[dr]_{q_{1}} & V_{2}\ar[l]_{f}\ar[r]^{g}\ar[d]^{q_{2}} & V_{3}\ar[dl]^{q_{3}}\\
 & X
}
\]
with $q_{1},q_{2},q_{3}$ plots and $f,g$ smooth maps, and $v_{1}\in\FV_{1},v_{2}\in\FV_{2},v_{3}\in\FV_{3}$
are fixed points such that $\Ff(v_{2})=v_{1}$ and $\Fg(v_{2})=v_{3}$.
Since $X=\colim_{U\in\mathcal{B}}U$, there exist open neighborhoods
$V_{1}'$ and $V_{3}'$ of $\st{v_{1}}$ and $\st{v_{3}}$ in $V_{1}$
and $V_{3}$, respectively, objects $p_{1}:U_{1}\ra X$ and $p_{3}:U_{3}\ra X$
in $\mathcal{B}$ and smooth maps $h_{1}:V_{1}'\ra U_{1}$ and $h_{3}:V_{3}'\ra U_{3}$
such that $p_{1}\circ h_{1}=q_{1}|_{V_{1}'}$ and $p_{3}\circ h_{3}=q_{3}|_{V_{3}'}$.
Write $V_{2}':=f^{-1}(V_{1}')\cap g^{-1}(V_{3}')$. It is clear that
$p_{1}(U_{1})\cap p_{3}(U_{3})\neq\emptyset$. By the assumption of
the proposition, there exist an object $r:W\ra X$ with $r(W)=p_{1}(U_{1})\cap p_{3}(U_{3})$
and morphisms $r\ra p_{1}$ and $r\ra p_{3}$ in $\mathcal{B}$. So
eventually we have the following commutative diagram in $\Diff$:
\[
\xymatrix{U_{1} & W\ar[l]\ar[r] & U_{3}\\
V_{1}'\ar[u]^{h_{1}}\ar@{^{(}->}[d] & V_{2}'\ar[u]\ar[l]\ar[r]\ar@{^{(}->}[d] & V_{3}'\ar@{^{(}->}[d]\ar[u]_{h_{3}}\\
V_{1} & V_{2}\ar[l]^{f}\ar[r]_{g} & V_{3}.
}
\]

From these diagrams, we know that the natural map $\colim_{U\in\mathcal{B}}\FU\ra\FX$
is injective.\end{proof}
\begin{example}
\label{exa:irrational-torus}Let $X$ be the $1$-dimensional irrational
torus of slope $\theta$, i.e., $X$ is the quotient group $\R/(\Z+\theta\Z)$
with the quotient diffeology, where $\theta$ is a fixed irrational
number. Let $\mathcal{J}$ be the category associated to the additive
group $\Z+\theta\Z$, i.e., $\mathcal{J}$ has one object, the morphisms
in $\mathcal{J}$ are indexed by the set $\Z+\theta\Z$, and the composition
corresponds to the addition in the additive group $\Z+\theta\Z$.
There is a functor $\mathcal{J}\ra\Diff$ sending $a+\theta b:\cdot\ra\cdot$
to $\R\ra\R$ with $x\mapsto x+(a+\theta b)$. It is straightforward
to check that the colimit of this functor is $X$. Since the projection
$\pi:\R\ra X$ is a diffeological covering (see~\cite[Chapter~8]{Igl13}),
one can show that the injectivity of Proposition~\ref{prop:colimit}
holds, and hence $\FX=\colim_{\mathcal{J}}\FR$, or more precisely,
$\FX$ is the quotient group $\FR/(\Z+\theta\Z)$.
\end{example}
However, the adding infinitesimal functor does not always preserve
colimits. That is why the calculation of $\FX$ is not easy. In particular,
this implies that the adding infinitesimal functor does not have a
right adjoint.
\begin{example}
\label{exa:not-commute-colimits}Let $R$ be the category associated
to the additive group $\R$, and let $F:R\ra\Diff$ be the functor
sending the object in $R$ to $\R$ and sending the morphism $r\in\R$
to the translation $\R\ra\R$ by $x\mapsto x+r$. Then $\colim F=\R^{0}$.
One can easily check that $|\colim({^{\bullet}(\blank)}\circ F)|=|D_{\infty}|$,
since there are only translations by reals, but $|^{\bullet}(\colim F)|=|\FR^{0}|=|\R^{0}|$.
Therefore, $\colim({^{\bullet}(\blank)}\circ F)\neq{^{\bullet}(\colim F)}$.
\end{example}

\subsubsection{Calculations of $\LY$}

In previous subsections, we have already calculated some examples
of $\LY$, where $Y$ is a Fermat space: in Subsection~\ref{sub:delete},
we showed that $_{\bullet}(\FX)=X$ and $_{\bullet}(\FDiff(\FX,\FZ))=\Diff(X,Z)$
for any diffeological spaces $X$ and $Z$, and in Subsection~\ref{sub:why},
we showed that $_{\bullet}(U\leq\FU)$ is a discrete Fermat space
for any open subset $U$ of a Euclidean space. We will calculate one
more example below, which will be useful for defining tangent spaces
and tangent bundles for Fermat spaces and diffeological spaces (which
is different from the approaches presented in~\cite{Ch-Wu16}) in
a future paper.
\begin{example}
\label{ex:calculate} Let $A$ be an ideal of $\FR$. Then $A\subseteq D_{\infty}$.
(See Section~\ref{sec:BasicFermatReals} for all possible expressions
of $A$, which are not needed in this example.) Let $Y$ be the quotient
ring $\FR/A$, equipped with the quotient Fermat space structure from
$\FR$. Then $\LY$ is diffeomorphic to $\R$. Here is the proof.
Note that the quotient map $\pi:\FR\ra Y$ induces a smooth map $_{\bullet}\pi:\R\ra\LY$.
By Theorem~\ref{thm:general}, we are left to show that $_{\bullet}\pi$
is injective. Assume that $x,y\in\R$ are mapped to the same point
in $\LY$. So there is a zig-zag diagram in $\F'/Y$ connecting two
copies of $\pi:\FR\ra Y$, such that after applying the deleting infinitesimal
functor on the zig-zag, there is a fixed point on $U$ for each $\FU\ra Y$
in the original zig-zag so that $x$ and $y$ gets connected by the
new zig-zag via these points. We break the original zig-zag into small
pieces as follows and study them to get information of the new zig-zag
on the corresponding small pieces:

(1) Assume that we have a commutative triangle 
\[
\xymatrix{\FR\ar[rr]^{f}\ar[dr]_{\pi} &  & \FU\ar[dl]^{p}\\
 & Y
}
\]
in $\F'/Y$, and points $x\in\R$ and $u\in U$ such that $\Lf(x)=u$.
Since $Y$ is a quotient Fermat space of $\FR$, there is a Fermat
open neighborhood $\FV$ of $f(x)$ in $\FU$ and a quasi-standard
smooth map $g:\FV\ra\FR$ so that $p|_{\FV}=\pi\circ g$. So we get
a commutative square 
\[
\xymatrix{ & \FW\ar[dl]\ar[dr]\\
\FR\ar[dr]_{\pi} &  & \FR\ar[dl]^{\pi}\\
 & Y,
}
\]
where $\FW$ is a Fermat open neighborhood of $x$ in $f^{-1}(\FV)$.
We will deal with this situation in (2).

(2) Assume that we have a commutative square 
\[
\xymatrix{ & \FU\ar[dl]_{f}\ar[dr]^{g}\\
\FR\ar[dr]_{\pi} &  & \FV\ar[dl]^{p}\\
 & Y
}
\]
in $\F'/Y$, and points $x\in\R$, $u\in U$ and $v\in V$ such that
$\Lf(u)=x$ and $\Lg(u)=v$. By a similar argument as (1), we may
assume that $\FV=\FR$ and $p=\pi$. By the commutativity of the square,
we have $f(u)-g(u)\in A$, which implies that $x=\Lf(u)=\Lg(u)=v$.

Together with these, we can conclude that $x=y\in\R$, i.e., the map
$_{\bullet}\pi:\R\ra\LY$ is injective, and hence a diffeomorphism.\end{example}

\end{document}